\documentclass[12pt]{amsart}
\usepackage{fullpage}
\usepackage[all]{xy}
\usepackage{amscd, amssymb, amsmath, amsthm, graphics}
\usepackage{amsmath,amsfonts,amsthm,amssymb}
\usepackage{latexsym,amsmath}
\usepackage{graphicx,psfrag}
\usepackage{pinlabel}
\usepackage{mathrsfs}
\usepackage{xcolor}
\usepackage{hyperref}

\newtheorem{theorem}{Theorem}[section]
\newtheorem{proposition}[theorem]{Proposition}

\newtheorem{lemma}[theorem]{Lemma}

\theoremstyle{definition}

\newtheorem{example}[theorem]{Example}

\newtheorem{remark}[theorem]{Remark}
\newtheorem{problem}[theorem]{Problem}

\numberwithin{equation}{section}
 \newcommand{\bZ}{\mathbb Z}

 \newcommand{\bQ}{\mathbb Q}

\begin{document}
\sloppy

\title[Strongly chiral 
 $\bQ$HS with hyperbolic fundamental groups]
 {Strongly chiral rational homology spheres with hyperbolic fundamental groups}

\author{Christoforos Neofytidis}
\address{Department of Mathematics and Statistics, University of Cyprus, Nicosia 1678, Cyprus}
\email{neofytidis.christoforos@ucy.ac.cy}

\subjclass[2010]{55M25, 57N65, 57M05, 20F34}
\keywords{}

\date{\today}

\begin{abstract}
For each $m\geq0$ and any prime $p\equiv3\ \mathrm{(mod \ 4)}$, we construct strongly chiral rational homology $(4m+3)$-spheres, 
 which have real hyperbolic fundamental groups and only  
non-zero integral intermediate homology groups isomorphic to $\bZ_{2p}$ in degrees $1,2m+1$ and $4m+1$. This gives group theoretic analogues in high dimensions of the existence of strongly chiral hyperbolic rational homology $3$-spheres, as well as 
of the existence of strongly chiral hyperbolic manifolds of any dimension that are not rational homology spheres, which was shown by Weinberger. One of our tools will be $r$-spins. We thus investigate the relationship between the sets of degrees of self-maps of a given manifold and its $r$-spins, and give classes of manifolds for which the sets are equal. 
\end{abstract}

\maketitle
\vspace{-.5cm}

\section{Introduction}

Let $M$ be a closed oriented $n$-manifold with fundamental class $[M]\in H_n(M;\bZ)$. A continuous map $f\colon M\to M$ is said to be of degree $d\in\bZ$, denoted by $\deg(f)=d$, if $H_n(f)([M])=d\cdot[M]$, and we define the {\em set of degrees of self-maps} of $M$ by
\[
D(M) := \{d\in\bZ \ | \  f\colon M\to M, \ \deg(f)=d\}.
\]

The study of $D(M)$ is a long-standing theme in topology. Clearly, we always have that $0\in D(M)$ (e.g. via a constant map) and $1\in D(M)$ (e.g. via the identity map), that is, $\{0,1\}\subseteq D(M)$ for any $M$. Naturally, two problems of particular interest are the following for a given manifold $M$: 

\begin{problem}\label{p1}
Is there $-1\in D(M)$?
\end{problem}

\begin{problem}\label{p2}
Is there $d\in D(M)$ with $|d|>1$?
\end{problem}

Problem \ref{p1} is part of the classic question for the existence or not of orientation-reversing maps, and has been explicitly formulated in this form by Sasao~\cite[Problem 2]{Sasao} within a set of problems aiming to understand how $D(M)$ characterises $M$; see~\cite[page 1]{Sasao}. Concerning Problem \ref{p2}, we note that the existence of a map of degree $d$ with $|d|>1$ implies immediately that $D(M)$ is infinite. Strong obstructions to this phenomenon have been developed through the concept of volume-type semi-norms in homology, most notably the simplicial volume introduced by Gromov~\cite{Gr82}. Hyperbolic manifolds consist a prominent class of spaces with positive simplicial volume~\cite{Gr82}, thus satisfying $D(M)\subseteq\{-1,0,1\}$. The later conclusion is more generally true for closed aspherical manifolds with non-elementary hyperbolic fundamental groups, however, it should be noted that there are classes of aspherical manifolds with zero or unknown simplicial volume but still satisfying $D(M)\subseteq\{-1,0,1\}$; see~\cite{NeoHopf,Neoendo} and their references. Hence, while Problem \ref{p2} is completely settled for hyperbolic manifolds, Problem \ref{p1} seems to be a more subtle one. 

In dimension two, every closed hyperbolic surface has self-mapping degree set equal to $\{-1,0,1\}$. However, for all $n\geq3$, S. Weinberger observed 
that a result of Belolipetsky and Lubotzky \cite[Theorem 1.1]{BL} implies that there is an abundance of hyperbolic $n$-manifolds with no self-maps of degree $-1$. We call such manifolds {\em strongly chiral} from corresponding concepts of mathematical biology and chemistry, most notably the notions of chiral knots and chiral molecules (the term of {\em chiral} alone in mathematics is reserved for the non-existence of orientation-reversing homeomorphisms).

\begin{theorem}[Weinberger]\cite[Section 3]{Mue}\label{t:W}
For each $n\geq3$, there are infinitely many strongly chiral closed hyperbolic $n$-manifolds.
\end{theorem}

Hyperbolic manifolds have usually rich, not sufficiently understood homology. Thus it would be natural to seek to refine Theorem \ref{t:W} into a more specific form by controlling, and if possible eliminating, the (rational) homology of the hyperbolic manifolds in all intermediate degrees. However, it is difficult to construct hyperbolic manifolds that are homology spheres, integrally or rationally. Indeed, this is impossible in dimensions $4m+2$, since an even dimensional rational homology sphere has Euler characteristic $\chi=2$, but the Euler characteristic of hyperbolic manifolds of even dimensions $n$ satisfies $(-1)^{n/2}\chi>0$; see~\cite[Theorem 1.62]{Lu}. In fact, the only known examples of hyperbolic rational homology spheres are in dimension three, where one can use Dehn surgery to construct infinitely many closed hyperbolic 
(rational or integral) homology 3-spheres. 
(Furthermore, Ratcliffe-Tschantz \cite{RT} constructed infinitely many aspherical homology $4$-spheres, which are however not hyperbolic.) Thus, currently it seems  out of reach an attempt to refine further Theorem \ref{t:W} by constructing (high dimensional) strongly chiral hyperbolic rational homology spheres, necessarily of odd dimension or of dimension divisible by four. Our goal in this paper is to show that we can 
at least completely control the fundamental group and the homology of the manifold:

\begin{theorem}\label{t:main}
For each $m\geq0$ and any prime $p$ so that $p\equiv 3$ $(\mathrm{mod} \ 4)$, there are infinitely many strongly chiral rational homology $(4m+3)$-spheres $M_m$ with 
real hyperbolic fundamental group and homology
\begin{equation}\label{eq.homology}
H_i(M_m;\bZ) =\left\{\begin{array}{ll}
        \bZ, & \text{for } i=0,4m+3\\
        \bZ_{2p}, & \text{for } i=1,2m+1,4m+1\\
        0, & \text{otherwise}.
        \end{array}\right.
\end{equation}  
\end{theorem}

Here, real hyperbolic fundamental group means the fundamental group of a closed real hyperbolic manifold (which will be a fixed $3$-manifold, as we shall see). 
Note that the homology of $M_m$ clearly generalises that of hyperbolic rational homology $3$-spheres obtained by $(2p,q)$-Dehn surgery on hyperbolic link complements in $S^3$.

The proof of Theorem \ref{t:main} will use as one of the building factors certain strongly chiral simply connected rational homology spheres, also used by M\"ullner~\cite{Mue}. In fact, based on the techniques and results of~\cite{Mue} and~\cite{Ne18}, one can construct many strongly chiral product manifolds with torsion-free hyperbolic fundamental groups (see Theorem \ref{t:products}). However, the disadvantage of taking products is that one looses the property of the manifolds being rational homology spheres. Therefore one needs to come up with a construction that does not produce new intermediate homology. The idea here is to combine the aforementioned simply connected manifolds with  spins of hyperbolic $3$-manifolds that are rational homology spheres, which will keep the integral homology controlled, namely $2p$-torsion, in three degrees: $1$, $2m+1$ and $4m+1$. (The definition of an $r$-spin will be given in Section \ref{s:spin}.) As indicated above, this construction will moreover have the strong feature that the fundamental groups  will remain fixed in all dimensions, giving us in particular a group theoretic analogue of Theorem \ref{t:W}. In fact, if one is willing to fix the fundamental group only in dimensions $\geq7$ by considering spins of a hyperbolic integral homology $3$-sphere rather than  the rational homology $3$-sphere $M_0$ of Theorem \ref{t:main}, then one could obtain a further refinement of Theorem \ref{t:main}, where the only non-zero intermediate homology $\bZ_{2p}$ will occur in degree $2m+1$; see Theorem \ref{t:main2}. 

The instrumental role of spins in the proof of Theorem \ref{t:main} motivates further research on understanding the relationship between $D(M)$ and $D(\sigma_r(M))$, where $\sigma_r(M)$ denotes the $r$-spin of $M$. As we shall see, while in general one expects that $D(\sigma_r(M))$ is larger than $D(M)$, it would be natural to understand when the two sets actually coincide; see Problem \ref{p4}. In this direction, we indicate the following:

\begin{theorem}\label{t:classes}
If $M$ is a sphere, a product of spheres, or a connected sum of products of spheres, then $D(M)=D(\sigma_r(M))$ for any $r>0$.
\end{theorem}

In the spirit of our main Theorem \ref{t:main} about controlling both the fundamental group and the homology of the strongly chiral manifolds, S. Weinberger pointed out to me a -- less explicit than Theorem \ref{t:main} -- way of obtaining strongly chiral homology spheres with fixed fundamental group. Namely, a result of Kervaire~\cite{K} says that for any finitely presented superperfect group $G$, i.e., satisfying $H_i(BG; \bZ) = 0$ for $i = 1, 2$, there is for any $n > 4$ a closed integral homology $n$-sphere $M$ with $\pi_1(M)= G$. In fact, if  $BG$ is a finite complex, then the semi-$h$-cobordism classes of such homology spheres are in a 1-1 correspondence with $\pi_1(BG^{+})$. With sufficient control of the automorphism group of $G$, one then can conclude that the homology spheres are moreover strongly chiral. This circle of ideas occurs in a paper by Nabutovsky and Weinberger~\cite{NW}, as well as in Weinberger's book~\cite{W}.

\subsection*{Outline of the paper}
In Section \ref{s:preliminaries} we will discuss the preliminaries that we need around the linking form and the spinning process. In Section \ref{s:proof} we will prove Theorem \ref{t:main}. Finally, in Section \ref{s:further} we will discuss some other constructions, and investigate the relationship between the sets of degrees of self-maps of a closed manifold and its $r$-spins, proving in particular Theorem \ref{t:classes}.

\subsection*{Acknowledgments}
 I would like to thank Joan Porti for his comments on the isometry groups of hyperbolic rational homology $3$-spheres, Shmuel Weinberger for his comments on an earlier draft of this paper, as well as an anonymous referee whose remarks helped improve the exposition. 
Also, I would like to thank 
Max Planck Institute for Mathematics in Bonn, where part of this work was carried out.

\section{Linking form and Spinning}\label{s:preliminaries} 
In this section we will give the background on the two building factors of our construction. 

\subsection{The linking form}
The reader can find a detailed exposition of the material of this paragraph in M\"ullner's thesis~\cite{Muethesis}.

\medskip

Given a closed oriented $(2k+1)$-manifold, $k\geq 1$, the {\em linking form} on $M$ is a non-degenerate $(-1)^{k+1}$-symmetric bilinear form
\begin{equation}\label{eq:linking}
L\colon \mathrm{Tor}H^{k+1}(M)\times\mathrm{Tor}H^{k+1}(M)\to \bQ/\bZ,
\end{equation}
which is roughly constructed as follows: Using the Universal Coefficient Theorem projection
\[
H^k(M;\bQ/\bZ)\to\mathrm{Tor} H^{k+1}(M)
\]
 and the (Kronecker) isomorphism (since $\bQ/\bZ$ is a divisible $\bZ$-module) 
\[
H^{k}(M;\bQ/\bZ)\to\mathrm{Hom}(H_k(M),\bQ/\bZ)
\]
one can construct a unique isomorphism
\begin{equation}\label{eq:linkingpart1}
\phi\colon \mathrm{Tor} H^{k+1}(M)\to \mathrm{Hom}( \mathrm{Tor}H_k(M),\bQ/\bZ).
\end{equation}
Also, by Poincar\'e Duality and the $5$-lemma, there is an isomorphism
\[
\delta\colon \mathrm{Tor} H^{k+1}(M)\to \mathrm{Tor}H_k(M),
\]
thus a dual isomorphism
\begin{equation}\label{eq:linkingpart2}
\delta^*\colon\mathrm{Hom}(\mathrm{Tor}H_k(M),\bQ/\bZ)\to\mathrm{Hom}(\mathrm{Tor} H^{k+1}(M),\bQ/\bZ).
\end{equation}
The composition of \eqref{eq:linkingpart1} and \eqref{eq:linkingpart2}, gives us the first  insertion of \eqref{eq:linking}, i.e.
\[
x\to L(x, \cdot)
\]
Checking the $(-1)^{k+1}$-symmetry is a routine and we get in particular that 
\[
x\to L(\cdot,x)
\]
is an isomorphism. Finally, one can check that for any map $f\colon M\to N$ and any cohomology classes $x,y\in\mathrm{Tor}H^{k+1}(M)$ we have
\begin{equation}\label{degree-linking}
L(H^{k+1}(f)(x),H^{k+1}(f)(y))=\deg(f)L(x,y).
\end{equation}

The $(-1)^{k+1}$-symmetry of the linking form gives us two immediate structure theorems, depending on whether $k$ is odd or even, i.e., whether $M$ has dimension congruent to $3$ modulo $4$ or congruent to $1$ modulo $4$ respectively. 
In the first case, this can be used to exclude self-maps of degree $-1$:

\begin{proposition}\cite[Lemma 8]{Muethesis}\label{p:kodd}
Let $M$ be a closed oriented manifold of dimension $2k+1$. 
 If $k=2m+1$ and $\mathrm{Tor} H^{k+1}(M)=\bZ_q$ so that $-1\not\equiv a^2$ $(\mathrm{mod} \ q)$ for all $a\in\bZ$, then $M$ is strongly chiral. 
\end{proposition}
\begin{proof}
Suppose there is a map $f\colon M\to M$ with $\deg(f)=-1$. Let $x\in\mathrm{Tor} H^{k+1}(M)=\bZ_q$ be a generator and $H^{k+1}(f)(x)=ax$, $a\in\bZ_q$. By \eqref{degree-linking}, we have
\[
- L(x,x)=L(H^{k+1}(f)(x),H^{k+1}(f)(x))=a^2 L(x,x)
\]
This implies that $-1\equiv a^2$ $(\mathrm{mod} \ q)$ because the linking form in non-degenerate, which contradicts our assumption.
\end{proof}

In the second case, when $M$ has dimension congruent to $1$ modulo $4$, then the linking form is antisymmetric and thus cannot be used anymore as in Proposition \ref{p:kodd}:

\begin{proposition}\cite[Corollary 93]{Muethesis}\label{p:keven}
Let $q>2$ and $M$ be a closed oriented manifold of dimension $2k+1$. If $k=2m$, then $\mathrm{Tor} H^{k+1}(M)\neq\bZ_q$.\end{proposition}

An example that illustrates Proposition \ref{p:kodd} is given by lens spaces $L_p=S^{4m+3}/\bZ_p$, where $p$ is a prime such that $p\equiv 3$ $(\mathrm{mod} \ 4)$, as one can see using for instance Fermat's little theorem. Since finite groups are elementary hyperbolic, this already gives us examples of strongly chiral manifolds with elementary hyperbolic fundamental groups. Our goal is to construct examples with freely indecomposable, torsion-free, non-elementary hyperbolic fundamental groups which are closer to the concept of Theorem \ref{t:W}, that is, using fundamental groups of real hyperbolic manifolds. 

\subsection{Spinning}\label{s:spin}

The material of this paragraph can be found in Suciu's paper on iterated spinning~\cite{Su}.

Given a closed manifold $M$ of dimension $n$ and an integer $r>0$, the {\em $r$-spin} of $M$ is a closed $(n+r)$-manifold defined by
\[
\sigma_r(M) =\partial((M\setminus\mathring{D^n})\times D^{r+1}) 
=(M\setminus\mathring{D^n})\times S^{r}\cup_{S^{n-1}\times S^r}S^{n-1}\times D^{r+1}
\]

The process of $r$-spinning was introduced by Epstein~\cite{Ep} and Cappell~\cite{Ca}, generalising work by Artin~\cite{Ar}.

Two facts that are of great interest to us, concerning the homotopy and the homology of an $r$-spin, are the following:

\begin{lemma}\label{l:homotopy}
Let $M$ be a closed manifold of dimension at least three. For each $r$, there is an isomorphism 
$$\pi_1(M)\cong\pi_1(\sigma_r(M)).$$ 
\end{lemma}

\begin{lemma}\label{l:homology}
Let $M$ be a closed manifold. For each $r$, the $i$th-homology of the  $r$-spin $\sigma_r(M)$ is given by 
$$H_i(\sigma_r(M))\cong H_i(M\setminus\mathring{D^n})\oplus \tilde{H}_{i-r}(M).$$
\end{lemma}

Lemma \ref{l:homotopy} and Lemma \ref{l:homology} are consequences of the Van Kampen theorem and the Mayer-Vietoris sequence respectively. Note that Lemma \ref{l:homotopy} fails in dimension two because 
$$\sigma_1(S^1\times S^1)=(S^2\times S^1)\#(S^2\times S^1),$$ see~\cite[Lemma 1.3]{Su}.

\section{Proof of Theorem \ref{t:main}}\label{s:proof}

In dimension three, Theorem \ref{t:main} amounts to the existence of a hyperbolic rational homology $3$-sphere with first integral homology group isomorphic to $\bZ_{2p}$: 
Let $N_p$ be a $3$-manifold obtained by hyperbolic $(2p,q)$-Dehn filling (for sufficiently large $q$) on the complement of a hyperbolic knot in $S^3$. This is indeed a hyperbolic $3$-manifold with integral cohomology 
\begin{equation}\label{eq.cohomologyN_p}
H^i(N_p;\bZ) =\left\{\begin{array}{ll}
        \bZ, & \text{for } i=0,3\\
        \bZ_{2p}, & \text{for } i=2\\
        0, & \text{otherwise.}
        \end{array}\right.
\end{equation}
Hence, if we choose $p$ to be a prime so that 
$p\equiv 3 \ (\mathrm{mod} \ 4),$ then we conclude by Proposition \ref{p:kodd} that $N_p$ is strongly chiral. 
 
\begin{remark}\label{r:aspherical}
Alternatively, the full isometry group $\mathrm{Isom}(N_p)$ can be taken to be of odd order; see~\cite{PP}\footnote{I thank J. Porti for drawing my attention to~\cite{PP} and for the enlightening communication.}. In that case, the argument of Theorem \ref{t:W} applies: Suppose there is a map $f\colon N_p\to N_p$ with $\deg(f)\neq0$. Since $N_p$ is hyperbolic, it has non-zero simplicial (or hyperbolic) volume, hence $\deg(f)=\pm1$. In particular, the induced endomorphism $\pi_1(f)$ is surjective and thus an isomorphism, since fundamental groups of hyperbolic manifolds (or just closed $3$-manifold groups) are Hopfian. Mostow's rigidity then tells us that $f$ is homotopic to an isometry. Since $\mathrm{Isom}(N_p)$ is of odd order, we conclude that $\deg(f)=1$. Note that this argument implies moreover that $D(N_p)=\{0,1\}$.
\end{remark}

The above give us Theorem \ref{t:main} in dimension three, i.e., the case $m=0$:

\begin{theorem}\label{t:dim3}
For each prime $p$ with $p\equiv 3 \ (\mathrm{mod} \ 4)$, there are infinitely many strongly chiral hyperbolic rational homology $3$-spheres $M_{0}=N_p$ with integral homology 
\begin{equation}\label{eq.homologyN_p}
H_i(N_p;\bZ) =\left\{\begin{array}{ll}
        \bZ, & \text{for } i=0,3\\
        \bZ_{2p}, & \text{for } i=1\\
        0, & \text{otherwise.}
        \end{array}\right.
\end{equation}
\end{theorem}

Now, we will show that there are infinitely many strongly chiral closed oriented manifolds $M_m$ of dimension $4m+3$, $m\geq1$, which have real 
hyperbolic fundamental group and non-zero intermediate integral homology groups isomorphic to $\bZ_{2p}$ in degrees $1,2m+1$ and $4m+1$.

\medskip

We begin by constructing the first building factor. Let $N_p$ be any strongly chiral hyperbolic rational homology $3$-sphere as given by Theorem \ref{t:dim3}. For each $m\geq1$, the $4m$-spins of each $N_p$ are given by
\[
\begin{array}{rcl}
N_{4m+3,p}=\sigma_{4m}(N_p)&=&\partial((N_p\setminus\mathring{D^3})\times D^{4m+1})\\
&=&(N_p\setminus\mathring{D^3})\times S^{4m}\cup_{S^{2}\times S^{4m}}S^{2}\times D^{4m+1}
\end{array}
\]
and they are closed oriented $(4m+3)$-manifolds. By Lemma \ref{l:homotopy}, the manifolds $N_{4m+3,p}$ have fundamental groups
\begin{equation}\label{pi1}
\pi_1(N_{4m+3,p})\cong\pi_1(N_p),
\end{equation}
which are hyperbolic $3$-manifold groups.  
By Lemma \ref{l:homology}, the homology of each $N_{4m+3,p}$ is given by
\begin{equation}\label{homology1}
H_i(N_{4m+3,p})\cong H_i(N_p\setminus\mathring{D^3})\oplus \tilde{H}_{i-4m}(N_p).
\end{equation}
By  \eqref{eq.homologyN_p}, excision and the long-exact sequence of the pair $(N_p,N_p\setminus\mathring{D^3})$, we have that \eqref{homology1} becomes
\begin{equation}\label{homology2}
H_i(N_{4m+3,p};\bZ) =\left\{\begin{array}{ll}
        \bZ, & \text{for } i=0,4m+3\\
        \bZ_{2p}, & \text{for } i=1,4m+1\\
        0, & \text{otherwise}.
        \end{array}\right.
\end{equation}
In particular, $N_{4m+3,p}$ is a rational homology $(4m+3)$-sphere for each $m\geq1$.

\medskip

Next, we will use the linking form to obstruct degree $-1$ maps as in Proposition \ref{p:kodd}, keeping at the same time the rest of the homology of the manifold trivial, which will give us the second building factor: Let the sphere $S^{2m+2}$. Then the tangent bundle $TS^{2m+2}$ has Euler class $$e(TS^{2m+2})=2\omega_{S^{2m+2}},$$ where $\omega_{S^{2m+2}}$ is a cohomological orientation (fundamental) class in $H^{2m+2}(S^{2m+2};\bZ)$; see for example~\cite[Prop. 3.14]{Ha}.

We pull back the bundle map $TS^{2m+2}\to S^{2m+2}$ by a self-map $f\colon S^{2m+2}\to S^{2m+2}$ to obtain a $D^{2m+2}$-bundle $f^*(TS^{2m+2})\to S^{2m+2}$ with Euler class 
$$e(f^*(TS^{2m+2}))=2\deg(f)\omega_{S^{2m+2}}.$$
 The cohomology of the associated $S^{2m+1}$-bundle $E_m$ over $S^{2m+2}$ can be computed via the Gysin sequence and is given by
\begin{equation}\label{cohomEm}
H^i(E_m;\bZ)=\left\{\begin{array}{ll}
        \bZ, & \text{for } i=0,4m+3\\
        \bZ_{2|\deg(f)|}, & \text{for } i=2m+2\\
        0, & \text{otherwise}.
        \end{array}\right.
\end{equation}
Hence, if we take $\deg(f)$ to be a prime number $p$ with 
$p\equiv 3 \ (\mathrm{mod} \ 4),$ then Proposition \ref{p:kodd} implies that $E_m$ is strongly chiral:

\begin{theorem}\label{t:scsc}
For each $m\geq0$, there is a $2m$-connected strongly chiral rational homology sphere $E_m$ of dimension $4m+3$. \end{theorem}

We remark that the $2m$-connectedness of $E_m$ follows from the homotopy long-exact sequence of the fibration $S^{2m+1}\to E_m\to S^{2m+2}$, namely
\[
\cdots\pi_{j+1}(S^{2m+2})\to\pi_j(S^{2m+1})\to\pi_j(E_m)\to\pi_j(S^{2m+2})\to\cdots.
\]
In particular, $E_m$ is simply connected for $m\geq1$. 

\begin{remark}
The manifolds $E_m$ were also used in~\cite[pp. 41-42]{Muethesis} and~\cite{Mue} for constructing strongly chiral products. Clearly, a construction using products can never preserve the property of being rational homology sphere. We will come back to the case of product manifolds in Section \ref{s:further}.
\end{remark}

With the two building factors in hands, we can now complete the proof of Theorem \ref{t:main}: For each $m\geq1$, let the connected sum
\[
M_m=E_m\# N_{4m+3,p}.
\]
Using \eqref{homology2} and \eqref{cohomEm}, we obtain that the homology of $M_m$ is given as in \eqref{eq.homology}; equivalently, its cohomology is
\begin{equation}\label{cohomMm}
H^i(M_m;\bZ)=\left\{\begin{array}{ll}
        \bZ, & \text{for } i=0,4m+3\\
        \bZ_{2p}, & \text{for } i=2,2m+2,4m+2\\
        0, & \text{otherwise}.
        \end{array}\right.
\end{equation}
Since $p\equiv 3 \ (\mathrm{mod} \ 4),$ we conclude again using Proposition \ref{p:kodd} that $M_m$ is strongly chiral. Moreover, since $E_m$ is simply connected, we conclude that
$$\pi_1(M_m)\cong\pi_1(N_{4m+3,p})\cong\pi_1(N_p),$$
which is a hyperbolic $3$-manifold group.

\medskip

The proof of Theorem \ref{t:main} is now complete.

\medskip

One could slightly refine the homological part of Theorem \ref{t:main} into the following form, where the price to be paid is that we need to use a different hyperbolic $3$-manifold for the spinning process than that of Theorem \ref{t:dim3}, and, in particular, the fundamental groups of $M_m$ are fixed only in dimensions $\geq7$:

 \begin{theorem}\label{t:main2}
For each $m\geq0$ and any prime $p$ so that $p\equiv 3$ $(\mathrm{mod} \ 4)$, there are infinitely many strongly chiral rational homology $(4m+3)$-spheres $M_m$ with real hyperbolic fundamental groups and homology
\begin{equation}\label{eq.homology2}
H_i(M_m;\bZ) =\left\{\begin{array}{ll}
        \bZ, & \text{for } i=0,4m+3\\
        \bZ_{2p}, & \text{for } i=2m+1\\
        0, & \text{otherwise}.
        \end{array}\right.
\end{equation}  
\end{theorem}
\begin{proof}
For $m=0$, we use the hyperbolic rational homology $3$-spheres  $M_0=N_p$ from Theorem \ref{t:dim3}. For $m\geq1$, we do not do $4m$-spins of $M_0$, but of a closed hyperbolic  integral homology $3$-sphere. The proof then follows verbatim that of Theorem \ref{t:main}.  The fundamental group, however, remains unchanged only in dimensions $\geq7$.
\end{proof}

Finally, we remark that one could also replace $\bZ_{2p}$ in Theorem \ref{t:main} (or in Theorem \ref{t:main2}) with other finite cyclic groups, as long as the conditions of Proposition \ref{p:kodd} are satisfied.

\section{Products, Spinning and Questions}\label{s:further}

We will now discuss some related constructions, as well as some questions on the ingredients of the main results of this paper.

\subsection{Products}

Exploiting Theorem \ref{t:scsc}, one can proceed using methods of~\cite{Mue} and~\cite{Ne18} to construct many strongly chiral manifolds with real hyperbolic fundamental groups, which are however not rational homology spheres:

\begin{theorem}\label{t:products}
For each $k\geq2$, there are infinitely many strongly chiral manifolds of dimension $4k+2$ with real hyperbolic fundamental group.
\end{theorem}

\begin{proof}
We know by Theorem \ref{t:W} that there are infinitely many strongly chiral hyperbolic $3$-manifolds $M$. For each $m\geq1$, let $E_m$ be a $2m$-connected strongly chiral rational homology sphere of dimension $4m+3$, given by Theorem \ref{t:scsc}. Then the conditions of~\cite[Theorem 1.4]{Ne18} apply and we deduce that the products $M\times E_m$ are strongly chiral.

Alternatively, we can fix a strongly chiral hyperbolic $3$-manifold $M$ (given by Theorem \ref{t:W}) and take the products with infinitely many $2m$-connected strongly chiral rational homology spheres $E_m$ of dimension $4m+3$, $m\geq1$, given by Theorem \ref{t:scsc}, which can be distinguished by their Euler class since $D(S^{2m+2})=\bZ$. Again, the conditions of~\cite[Theorem 1.4]{Ne18} apply and thus $M\times E_m$ are strongly chiral.
\end{proof}

Clearly, the products $M\times E_m$ are not rational homology spheres. We note that the second proof of Theorem \ref{t:products} is closer to the spirit of Theorem \ref{t:main}. Indeed, as explained in the proof of Theorem \ref{t:main}, the strongly chiral hyperbolic $3$-manifolds factors $M$ can be taken to be rational homology spheres (Theorem \ref{t:dim3}), giving us thus complete control of the homology. However, even with the latter reduction, we would still need to kill the $\bZ$ factor in the third homology group coming from $M$. This is clearly impossible by taking products. The spinning process gave us a delicate way to avoid this issue.

\subsection{Self-mapping degrees of $r$-spins}

The argument in Remark \ref{r:aspherical} cannot be used anymore for higher dimensional $r$-spins, since the fundamental group will remain unchanged, more precisely, $3$-dimensional hyperbolic, and, in particular, the $r$-spins will not be aspherical. Therefore, an interesting question would be the following:

\begin{problem}\label{p3}
What is the relationship between $D(M)$ and $D(\sigma_r(M))$?
\end{problem}

The first natural case to consider is that of a sphere, for which we have
\[
D(S^n)=D(\sigma_r(S^n)),
\]
since $\sigma_r(S^n)=S^{n+r}$ for any $r>0$. Next, one can consider products of spheres. In that case, we have $D(S^n\times S^m)=\bZ$ and the $r$-spin of $S^n\times S^m$ is given by
\[
\sigma_r(S^n\times S^m)=(S^{n+r}\times S^m)\#(S^n\times S^{m+r});
\]
see~\cite[Lemma 1.3]{Su}. We have the following:

\begin{lemma}\label{degreesrspin}
$D(\sigma_r(S^n\times S^m))=\bZ$ for all $r>0$.
\end{lemma}
\begin{proof}
Since $\sigma_r(S^n\times S^m)=(S^{n+r}\times S^m)\#(S^n\times S^{m+r})$, we observe that the two factors $S^{n+r}$ and $S^{m+r}$ in $S^{n+r}\times S^m$ and $S^n\times S^{m+r}$ respectively are simply connected. On both of  those simply connected factors, we take a self-map $g_i$ of any degree $d$. On the other factor of each summand, $S^m$ and $S^n$ (which could be $S^1$), we take the identity. The product self-maps on each summand 
$$f_1=g_1\times id\colon S^{n+r}\times S^m\to S^{n+r}\times S^m$$
and
 $$f_2=id\times g_2\colon S^n\times S^{m+r}\to S^n\times S^{m+r}$$
  have therefore degree $d$ and they are moreover $\pi_1$-surjective, since on the simply connected factors of each summand we have the self-maps of degree $d$, and on the other factors the identity which induces an isomorphism on $\pi_1$. Then, by~\cite{RW}, we homotope each $f_i$ so that the inverse image of some $(m+n+r)$-ball $D_i$ in the codomain is an $(m+n+r)$-ball $D_i'$ in the domain. Hence, we obtain maps
  $$\overline{f_1}\colon (S^{n+r}\times S^m)\setminus D_i'\to (S^{n+r}\times S^m)\setminus D_i$$
and
 $$\overline{f_2}\colon (S^n\times S^{m+r})\setminus D_i'\to (S^n\times S^{m+r})\setminus D_i$$
  which induce self-maps of degree $d$ on the boundary $(m+n+r-1)$-spheres. Then we glue $\overline{f_1}$ and $\overline{f_2}$ along these boundaries to obtain a map
  \[
  (S^{n+r}\times S^m)\#(S^n\times S^{m+r})\to (S^{n+r}\times S^m)\#(S^n\times S^{m+r})
  \]
  of degree $d$.
\end{proof}

In particular, Lemma \ref{degreesrspin} tells us that
\[
D(S^n\times S^m)=D(\sigma_r(S^n\times S^m)),
\]
for all $r>0$.

However, in general, the two sets $D(M)$ and $D(\sigma_r(M))$  are not related as shown by the following examples. 

\begin{example}\label{ex1}
Starting already from hyperbolic surfaces 
$\Sigma_g$, one has 
$$D(\Sigma_g)=\{-1,0,1\}\neq\bZ,$$
while as we have mentioned above (see~\cite{Su})
\[
\sigma_r(\Sigma_g)=\#_{2g}(S^{r+1}\times S^1),
\]
which means that
\[
D(\sigma_r(\Sigma_g))=\bZ, 
\]
by Lemma \ref{degreesrspin}. 
\end{example}

\begin{example}\label{ex2}
For the complex projective space $\mathbb{CP}^n$ we have (see~\cite[Example 1.6]{NSTWW})
\[
D(\mathbb{CP}^n)=\{k^n \ | \ k\in\bZ\}\neq\bZ, \ n\geq2,
\]
while  (see~\cite{Su})
\[
\sigma_r(\mathbb{CP}^n)=\mathbb{CP}^{n-1}\times S^{r+2},
\]
and so
\[
D(\sigma_r(\mathbb{CP}^n))=\bZ.
\]
\end{example}

Both examples suggest something natural, namely that $r$-spins have larger self-mapping degree sets than the original manifolds, since the $r$-spinning process involves spheres. We could thus refine Problem \ref{p3} as follows:

\begin{problem}\label{p4}
Classify those manifolds which satisfy $D(M)=D(\sigma_r(M))$.
\end{problem}

As a first observation regarding Problem \ref{p4} we have the following:

\begin{theorem}[Theorem \ref{t:classes}]
Suppose $M$ is either a sphere, a product of spheres or a connected sum of products of spheres. Then  $D(M)=D(\sigma_r(M))=\bZ$ for any $r>0$.
\end{theorem}
\begin{proof}
We have seen above the two cases
\[
D(S^n)=D(\sigma_r(S^n))=\bZ,
\]
and
\[
D(S^n\times S^m)=D(\sigma_r(S^n\times S^m))=\bZ.
\]
For the remaining case, let a closed $n$-manifold $M$ of the form
\[
M=\#_{i=1}^s(S^{k_i}\times S^{m_i}),
\]
where $k_i+m_i=n$ for all $i=1,...,s$. Since $\sigma_r(S^{k_i}\times S^{m_i})=(S^{k_i+r}\times S^{m_i})\#(S^{k_i}\times S^{m_i+r})$, and $\sigma_r(M_1\#M_2)=\sigma_r(M_1)\#\sigma_r(M_2)$ for any two manifolds $M_1,M_2$ (see~\cite{Go}), we obtain
\[
\sigma_r(M)=\#_{i=1}^s\sigma_r(S^{k_i}\times S^{m_i})=\#_{i=1}^s((S^{k_i+r}\times S^{m_i})\#(S^{k_i}\times S^{m_i+r})).
\]
Thus, Lemma \ref{degreesrspin} tells us that
\[
D(\sigma_r(M))=D(M)=\bZ.
\]
\end{proof}

\subsection{Iterated spinning}
Given an $m$-tuple $R=(r_1,r_2,...,r_m)$, $r_i>0$, the (iterated) $R$-spin of a closed $n$-manifold $M$ is given by
 \[
 \sigma_R=\sigma_{r_1}\sigma_{r_2}\cdots\sigma_{r_m}(M).
 \]
 This is a closed $(n+|R|)$-manifold, where $|R|=r_1+r_2+\cdots+r_m$. By Lemma \ref{l:homotopy}, it is clear that $\pi_1(M)\cong\pi_1(\sigma_R(M))$, for $n\geq3$. Moreover, $\sigma_R$ does not depend on the order of the $m$-tuple $R=(r_1,r_2,...,r_m)$, since
 \begin{equation}\label{orderedtuple}
 \sigma_{r_ir_j}(M)\cong\sigma_{r_jr_i}(M)
 \end{equation}
 by~\cite[Lemma 1.2]{Su}.

It is now tempting from the above iterated spinning process, that, if one is willing to sacrifice the homology models given by the manifolds $M_m$ in Theorem \ref{t:main}, then one could try to appeal to this process to construct strongly chiral rational homology spheres with hyperbolic fundamental groups just by applying Proposition \ref{p:kodd}, i.e., without using the $2m$-connected manifolds $E_m$. However, beyond that the iterated spinning process will produce models with more complicated homology, including unbounded homology as $m\to\infty$, it will not generally produce manifolds where Proposition \ref{p:kodd} will be directly applicable, as shown by the following table with all possible $7$-dimensional iterated spinnings produced from the rational homology $3$-spheres $M$ of Theorem \ref{t:dim3}.

\begin{table}[!ht]
\centering
{\small
\begin{tabular}{c|c}
Iterated spinning of $M$ &  $H_3(\cdot;\bZ)$\\
\hline
$\sigma_1\sigma_1\sigma_1\sigma_1(M)$  &    $\bZ_{2p}^6$\\
$\sigma_2\sigma_1\sigma_1(M)$        &     $\bZ_{2p}^2$\\
$\sigma_3\sigma_1(M)$ & $0$\\  
$\sigma_2\sigma_2(M)$ &  $\bZ_{2p}^2$\\
& 
\end{tabular}}
\centering
\caption{}
\end{table}

Note that the iterated spinnings $\sigma_1\sigma_2\sigma_1(M)$, $\sigma_1\sigma_1\sigma_2(M)$ and $\sigma_1\sigma_3(M)$ are not shown on the table, because
\[
\sigma_1\sigma_2\sigma_1(M)\cong\sigma_1\sigma_1\sigma_2(M)\cong\sigma_2\sigma_1\sigma_1(M)
\]
and
\[
\sigma_3\sigma_1(M)\cong\sigma_1\sigma_3(M),
\]
as we mentioned; cf.~\eqref{orderedtuple}.

In addition to the above, the discussion in the preceding paragraph indicates that in general more degrees of self-maps will occur for $\sigma_R(M)$ than for $M$.

\bibliographystyle{alpha}

\end{document}